\documentclass[reqno,12pt,letterpaper]{amsart}

\usepackage[usenames,dvipsnames]{color}
\usepackage{amsmath,amssymb,amsthm,graphicx,mathrsfs,url}
\usepackage{color,graphicx}
\usepackage{verbatim}
\usepackage{url}
\usepackage{graphicx}

\usepackage[colorlinks=true,linkcolor=Red,citecolor=Green]{hyperref}

\newcommand{\Id}{I}

\newcommand{\diam}{\operatorname{diam}}

\newcommand{\supp}{\operatorname{supp}}

\newcommand{\C}{\mathbb C}

\newcommand{\R}{\mathbb R}
\newcommand{\indic}{\operatorname{1\negthinspace l}}

\def\m{\mathbf{m}}
\def\s{\mathbf{s}}
\def\<{\langle}
\def\>{\rangle}
\newcommand{\bp}{{\it Proof. }}
\newcommand{\ep}{\hfill $\square$\\}

\newcommand{\be}{\begin{equation}}
\newcommand{\ee}{\end{equation}}
\newcommand{\bes}{\begin{equation*}}
\newcommand{\ees}{\end{equation*}}

\numberwithin{equation}{section}

\newtheorem{theorem}{Theorem}

\newtheorem{proposition}{Proposition}
\newtheorem{lemma}[proposition]{Lemma}

\def\aaa{{\mathcal A}}
\def\ggg{{\mathcal G}} 
 \def\lll{{\mathcal L}}
\def\mmm{{\mathcal M}} \def\ooo{{\mathcal O}}

\def\uuu{{\mathcal U}}

\begin{document}
\title{A semiclassical approach to the Kramers--Smoluchowski equation}

\author[L. Michel]{Laurent~Michel}
\email{lmichel@\allowbreak unice.\allowbreak fr}
\address{Universit\'e C\^ote d'Azur, CNRS, Laboratoire J.-A. Dieudonn\'e, France}
\author[M. Zworski]{Maciej Zworski}
\email{zworski@math.berkeley.edu}
\address{Department of Mathematics, University of California,
Berkeley, CA 94720, USA}

\def\arXiv#1{\href{http://arxiv.org/abs/#1}{arXiv:#1}}

\begin{abstract}
We consider the Kramers--Smoluchowski equation at a low temperature regime
and show how semiclassical techniques developed for the study of the 
Witten Laplacian and Fokker--Planck equation provide quantitative results. 
This equation comes from molecular dynamics and temperature plays the role of 
a semiclassical paramater.
The presentation is self-contained in the one dimensional case, with pointers to the recent paper \cite{Mi16} for results needed in higher
dimensions. One purpose of this note is to provide a simple introduction to semiclassical methods in this context. 
\end{abstract}

\maketitle
\section{Introduction}
The Kramers--Smoluchowski equation describes  
the time evolution of the probability density of a particle undergoing a Brownian motion under the influence of a chemical potential -- see \cite{Berg} for
the background and references.  
Mathematical treatments in the low temperature regime have been provided 
by Peletier et al \cite{Pele} using $ \Gamma$-convergence, by 
Herrmann--Niethammer \cite{herr} using {W}asserstein gradient flows and 
by Evans--Tabrizian \cite{EvTa16}. 

The purpose of this note is to explain how precise quantitative results can be 
obtained using semiclassical methods developed by, among others, Bovier, Gayrard, Helffer, H\'erau, Hitrik, Klein, Nier and Sj\"ostrand \cite{BoGaKl05_01,He88_01,HeKlNi04_01,HeSj85_01,HeHiSj11_01} for the study of spectral asymptotics for Witten Laplacians \cite{Wi82} 
and for Fokker--Planck operators. The semiclassical parameter $ h$ is 
the (low) temperature. This approach is much closer in spirit to the heuristic
arguments in the physics literature \cite{Ey,Kra} and the main point is that
the Kramers--Smoluchowski equation {\em is} the heat equation for the Witten 
Laplacian acting on functions.
Here 
we give a self-contained presentation of the one dimensional case 
and explain how the recent paper by the 
first author \cite{Mi16} can be used to obtain results in higher dimensions. 

Let $\varphi:\R^d\rightarrow \R$ be a smooth function.
Consider the corresponding Kramers-Smoluchowski equation:
\begin{equation}\label{eq:KS}
\left\{\begin{array}{c}
\partial_t\rho= \partial_x \cdot (\partial_x \rho+\epsilon^{-2}\rho \partial_x\varphi)\\
\rho_{\vert t=0}=\rho_0\phantom{********}
\end{array}
\right.
\end{equation}
where $\epsilon\in (0,1]$ denotes the temperature of the system and will be the small asymptotic parameter. Assume that 
there exists $C>0$ and a compact $K\subset\R^d$ such that for all $x\in\R^d\setminus K$, we have
\begin{equation}\label{eq:hyp-croissancephi}
 \vert\partial \varphi(x)\vert\geq\frac 1 C, \ \ \ \ \vert \partial^2_{x_i x_j}\varphi\vert\leq C\vert\partial\varphi\vert^2, \ \ \ \ \varphi(x)\geq C\vert x\vert.
\end{equation}
Suppose additionally that $\varphi$ is a Morse function, that is,
$ \varphi $ has isolated and non-degenerate critical points. Then, thanks to the above assumptions the set $\uuu$ of critical points of $\varphi$ is finite.
For $p=0,\ldots,d$, we denote by $\uuu^{(p)}$ the set of critical points of index $p$. 
Denote 
\begin{equation}\label{eq:defm0m1}
\varphi_0:=\inf_{x \in \R^d}\varphi(x)=\inf_{\m\in\uuu^{(0)}}\varphi(\m) \ \ \text{ and } \ \ \sigma_1:=\sup_{\s\in\uuu^{(1)}}\varphi(\s).
\end{equation}
Thanks to \eqref{eq:hyp-croissancephi}, the sublevel set of $ \sigma_1 $ is
decomposed in finitely many connected components $E_1,\ldots,E_N$:
\begin{equation}
\label{eq:En} 
 \{x\in\R^d,\,\varphi(x)<\sigma_1\} = \bigsqcup_{n=1}^N E_n .
\end{equation}
We assume that 
\begin{equation}\label{eq:hyptopol}
 \inf_{x\in E_n}\varphi(x)=\varphi_0, \ \ \  \forall n=1,\ldots, N, \ \ \text{ and }\ \
 \varphi (\s ) = \sigma_1 , \ \  \forall\s \in \mathcal U^{(1)} . 
\end{equation}
which corresponds to the situation where $\varphi$ admits $N$ wells of the same height. 
In order to avoid heavy notation, we also assume that for 
$n=1,\ldots,N$ the minimum of $\varphi$ on $E_n$ is attained in a single point that we denote by $\m_n$.

The  associated {\em Arrhenius number}, 
$S=\sigma_1-\varphi_0$,  governs the long time dynamics of \eqref{eq:KS}. That is made quantitative in Theorems \ref{th1} below. More general assumptions 
can be made as will be clear from the proofs. We restrict ourselves to the 
case in which the asymptotics are cleanest.

\begin{figure}
 \center
  \scalebox{0.6}{ 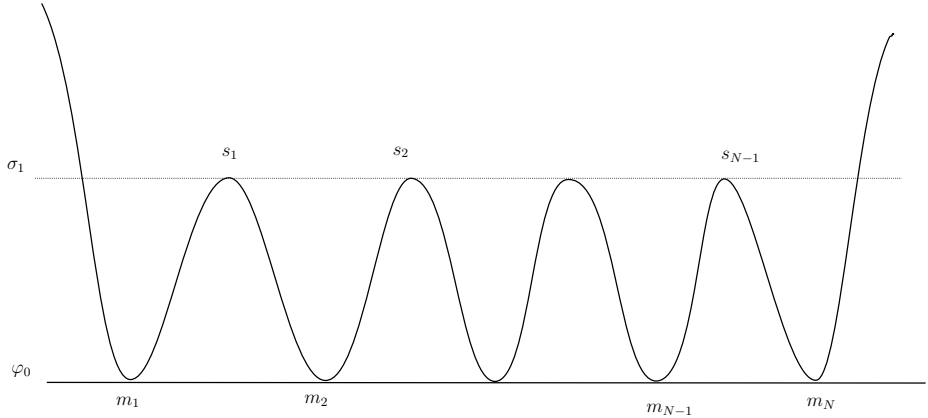}
  \caption{A one dimensional potential with interesting Kramers--Smoluchowski 
  dynamics.}
    \label{fig1}
\end{figure}

To state the simplest result let us assume that $ d = 1 $ and that the  second derivative of $\varphi$ is constant on the sets 
$\uuu^{(0)}$ and $\uuu^{(1)}$:  
\be\label{hyp:phisec-const}
\varphi''(\m)=\mu,\ \ \ \forall \m\in \uuu^{(0)}\ \ \ \text{ and }\ \ \ \varphi''(\s)=-\nu,\ \ \ \forall \s\in \uuu^{(1)}
\ee
for some $\mu,\nu>0$.
 The potential then looks like the one 
shown in Fig.\ref{fig1}. We introduce the matrix
\begin{equation}
\label{eq:simpleA0} A_0 =
\frac \kappa \pi
\begin{pmatrix} 1&\!\! \! - 1  &0&0&\ldots&\ldots&\ldots&0\\
\!\! \! - 1  &2&\!\! \! - 1  &0&\ldots&\ldots&\ldots&0\\
0&\!\! \! - 1  &2&\!\! \! - 1  &0&\ldots&\ldots&0\\
\vdots&0&\!\! \! - 1  &2&\ddots&\ddots&\ldots&0\\
\vdots&\vdots&\ddots&\ddots&\ddots&\ddots&\ddots&0\\
\vdots&\vdots&\ddots&\ddots&\ddots&\ddots&\!\! \! - 1  &0\\
0&\vdots&\ddots&\ddots&\ddots&\!\! \! - 1  &2&\!\! \! - 1  \\
0&0&\ldots&\ldots&\ldots&0&\!\! \! - 1  &1
\end{pmatrix}.
\end{equation}
with $ \kappa=\sqrt{\mu \nu} $.
This matrix is positive semi-definite with a simple eigenvalue at $ 0 $.

\begin{theorem}\label{th1}

Suppose that $ d = 1 $ and $ \varphi $ satisfies \eqref{eq:hyp-croissancephi},
 \eqref{eq:hyptopol} and \eqref{hyp:phisec-const}. 
 Suppose that 
 \begin{equation}
\label{eq:rho0}
 \rho_0=\left(\frac \mu{2\pi\epsilon^2}\right)^{\frac 12} \left( \sum_{n=1}^N \beta_n 
  \indic_{E_n} + \,  r_\epsilon \right) e^{-\varphi/\epsilon^2}   , \ \ \ 
 \lim_{ \epsilon \to 0 } \Vert r_\epsilon\Vert_{L^\infty} =  0,  \ \ \beta\in 
\R^N , 
 \end{equation} 
then the solution to \eqref{eq:KS} satisfies, uniformly for $ \tau \geq 0  $,\begin{equation}\label{eq:weakconv}
\rho( 2\epsilon^2e^{S/\epsilon^2}\tau,x)\ \rightarrow \ \sum_{n=1}^N\alpha_n(\tau)\delta_{\m_n} (x) , 
\ \ \ \epsilon \to 0, 
\end{equation}
in the sense of distributions in $ x $, where  $ S = \sigma_1 - \varphi_0 $ and 
where 
$\alpha(\tau)=(\alpha_1,\ldots,\alpha_n)(\tau)$ solves 
\begin{equation}
\label{eq:dotal} \partial_\tau \alpha= - A_0\alpha , \ \ \ \alpha ( 0 ) = \beta ,
\end{equation}
with $ A_0 $ given by \eqref{eq:simpleA0}.
\end{theorem}

The above result is a generalization of Theorem 2.5 in \cite{EvTa16} where the case of 
a double-well is considered and estimates are uniform on compact time intervals only. 
We remark that the equation considered in \cite{EvTa16} has also an additional transverse variable (varying slowly).
A development of the methods presented in this note would also allow having
such variables. Since our goal is to explain
general ideas in a simple setting we do not address this issue here. 

A higher dimensional version of Theorem \ref{th1} is given in Theorem \ref{t:higher}
in \S \ref{s:get_high}.  In this higher dimensional setting, the matrix $ A_0 $ becomes a graph Laplacian for a graph
obtained by taking minima as vertices and saddle points as edges. The same graph Laplacian was used by Landim et al \cite{Lamit} in 
the context of a discrete model of the Kramers--Smoluchowski equation.

Using methods of \cite{EvTa16} and \cite{BoGaKl05_01},
Theorem \ref{t:higher} was also proved by Seo--Tabrizian \cite{Into}, but as the other 
previous papers, without uniformity in time (that is, with convergence uniform for $ t 
\in [ 0 , T ] $).

Here, Theorem \ref{th1} is a consequence of a more precise asymptotic formula 
given in Theorem \ref{th:DynamWitten} formulated using the Witten Laplacian.
Provided that certain topological assumptions are satisfied (see \cite[\S 1.1, \S 1.2]{Mi16}) an analogue of Theorem \ref{th1} in higher dimensions is immediate -- 
see \S \ref{s:get_high} for geometrically interesting examples.

The need for the new results of \cite{Mi16} comes from the fact that
in the papers on the low-lying eigenvalues of the Witten Laplacian \cite{BoGaKl05_01,He88_01,HeKlNi04_01,HeSj85_01,HeHiSj11_01} the authors make assumptions on the relative positions of minima and of
saddle points. These assumptions mean that 
the Arrhenius numbers are distinct and hence potentials for which 
the Kramers--Smoluchowski dynamics \eqref{eq:weakconv} is interesting are excluded.
With this motivation the general case was studied in \cite{Mi16} and to explain 
how the results of that paper can be used in higher dimensions we give 
a self-contained presentation in dimension one.

We remark that we need specially prepared initial data \eqref{eq:rho0} to 
obtain results valid for all times. Also, $ E_n $'s in the 
statement can be replaced by any interval in $ E_n $ containing the minimum.
 Theorem \ref{th:DynamWitten} also shows that
a weaker result is valid for any $ L^2 $ data: suppose that
$ \rho_0 \in L^2_\varphi:= L^2(e^{\varphi(x)/\epsilon^2}dx)$ and that 
\[  \beta_n :=\left(\frac \mu{2\pi\epsilon^2}\right)^{\frac 14}   
\int_{E_n }  \rho_0  (x ) dx . \]
Then, uniformly for $ \tau \geq 0 $, 

\begin{equation}
\label{eq:rhot} 
\begin{gathered} \rho ( t, x ) = \left(\frac \mu{2\pi\epsilon^2}\right)^{\frac 14}   
\sum_{n=1}^N  \alpha_n ((2\epsilon^2)^{-1}e^{-S/\epsilon^2} t ) \indic_{E_n } ( x ) e^{-\varphi/\epsilon^2}+ 
r_\epsilon(t,x) ,\\
\Vert r_\epsilon(t)\Vert_{L^1(dx)}\leq C(\epsilon^{\frac 52}+\epsilon^{\frac 12}e^{-t\epsilon^2})\Vert\rho_0\Vert_{L^2_\varphi} , \ \ \
L^2 _\varphi := L^2 (\R , e^{\varphi(x)/\epsilon^2}dx) .
\end{gathered} 
\end{equation}
where $ \alpha $ solves \eqref{eq:dotal}. 
The proof of \eqref{eq:rhot}
 is given at the end of \S \ref{s:pft}.


\medskip

\noindent
{\sc Acknowledgements.} We would like to thank Craig Evans and 
Peyam Tabrizian for introducing us to the Kramers--Smoluchowski equation, and
Insuk Seo for informing us of reference \cite{Lamit}. 
The research of LM was partially supported by
the European Research Council, ERC-2012-ADG, project number 320845 and by the France Berkeley Fund. MZ acknowledges partial support under the National Science Foundation grant DMS-1500852.

\section{Dimension one}

In this section we assume that the dimension is equal to $d=1$. That allows to present self-contained proofs which indicate the
strategy for higher dimension.

Ordering the sets $E_n$ such that 
$\m_1<\m_2<\ldots<\m_N$ it follows that  for all $n=1,\ldots N-1$ $\bar E_n\cap \bar E_{n+1}=\{\s_n\}$ is a maximum and we assume additionally that there exists $\mu_n,\nu_k>0$ such that for
$ n=1,\ldots,N$ and $ k=1,\ldots,N-1 $, 
\begin{equation}\label{eq:nothessphi}
\varphi''(\m_n)=\mu_n \ \ \text{ and } \ \ \varphi''(\s_k)=-\nu_k.
\end{equation}
Using this notation we define a symmetric $N\times N$ matrix:
$ A_0= (a_{ij} )_{1 \leq i,j \leq N } $, where (with the convention that 
$ \nu_0 = \nu_N = 0 $)
\be\label{eq:defA0}
\begin{split} 
a_{ii} & = \pi^{-1}\mu_j^{\frac12} (\nu_{j-1}^{\frac12} + \nu_j^{\frac12} ) , 
\ \ \ a_{i,i+1} 
= -\pi^{-1} \nu_i^{\frac12} \mu_{i}^{\frac14} \mu_{i+1}^{\frac14}, \ \ 
1 \leq i \leq N-1 ,  
\end{split}
\ee
and $ a_{i,i+k}   = 0$  , for $k > 1$, $ a_{ij} = a_{ji}$. 
The matrix $A_0$ is symmetric positive and the eigenvalue $0$ has multiplicity $1$.
When $ \mu_j $'s and $ \nu_j $'s are all equal our matrix takes the particularly simple form \eqref{eq:simpleA0}.

First, observe that we can assume without loss of generality that $\varphi_0=0$.
Define  the operator appearing on the right hand side of \eqref{eq:KS} by
\[ P: =\partial_x \cdot (\partial_x+\epsilon^{-2}\partial_x\varphi)\]
and denote 
\[ h=2\epsilon^2 . \]
Then, considering $ e^{\pm \varphi/h } $ as a multiplication operator, 
$$P=\partial_x\circ(\partial_x+2h^{-1}\partial_x\varphi)=\partial_x \circ e^{-2\varphi/h}\circ\partial_x\circ  e^{2\varphi/h}$$
and  
$$e^{\varphi/h}\circ P\circ e^{-\varphi/h}=-h^{-2}\Delta_\varphi, \ \ \ 
\Delta_\varphi : =-h^2\Delta+\vert\partial_x\varphi\vert^2-h\Delta\varphi .$$
Hence, $\rho$ is solution of \eqref{eq:KS} if 
 $u(t,x):=e^{\varphi(x)/h}\rho(h^2t,x)$ is a solution of 
\begin{equation}\label{eq:heatwitten}
\partial_tu =-\Delta_\varphi u,  \ \ \ 
u_{\vert t=0}=u_0:=\rho_0e^{\varphi/h} .
\end{equation}
In order to state our result for this equation, we denote
\be\label{eq:defpsin}
\psi_n(x):=c_n(h)h^{-\frac 1 4}\indic_{E_n}(x)e^{-(\varphi-\varphi_0)(x)/h}, \ \ \ \forall n=1,\ldots,N , 
\ee
where $c_n(h)$ is a normalization constant such that $\Vert \psi_n\Vert_{L^2}=1$. The method of steepest descent shows that
\be\label{eq:asympcn}
c_n(h) \sim \sum_{k=0}^\infty c_{n,k}h^k, \ \ 
c_{n,0}=(\mu_n/\pi)^{\frac 14} , \ \  \forall n=1,\ldots N.
\ee
We then define a map $\Psi:\R^N\rightarrow L^2$ by 
\be\label{eq:defPsi}
\Psi(\beta):=\sum_{n=1}^N\beta_n\psi_n, \ \ \ \forall \beta = ( \beta_1 , \ldots , \beta_N ) 
\in\R^N. 
\ee
The following theorem describes the dynamic of the above equation when $h\rightarrow 0$.
\begin{theorem}\label{th:DynamWitten}
There exists $C>0$ and $h_0>0$ such that for all  $\beta\in\R^N$ and all $0 < h < h_0 $, we have
\be\label{eq:DynamWitten1}
\Vert e^{-t\Delta_\varphi}\Psi(\beta)-\Psi(e^{-t\nu_h A}\beta)\Vert_{L^2}\leq Ce^{-\frac 1{Ch}} | \beta | , \ \ \  \forall t\geq 0, 
\ee
where $\nu_h=he^{-2S/h}$, $ S = \sigma_1 - \varphi_0 $, and $A=A(h)$ is a real symmetric positive matrix having a classical expansion
$
A\sim \sum_{k=0}^\infty h^k A_k
$
with
$A_0$ given by \eqref{eq:defA0}. In addition,
\be\label{eq:DynamWitten2}
\Vert e^{-t\Delta_\varphi}\Psi(\beta)-\Psi(e^{-t\nu_h A_0}\beta)\Vert_{L^2}\leq Ch | \beta | 
\ee
uniformly with respect to $t\geq 0$.
\end{theorem}

We first show how 
\begin{proof}[Theorem \ref{th:DynamWitten} implies Theorem \ref{th1}]
First recall that we assume here $\mu_n=\mu$ for all $n=1,\dots N$ and $\nu_k=\nu$ for all $k=1,\ldots N-1$.
Suppose  that $\rho$ is the solution to \eqref{eq:KS} with 
$\rho_0$ as in Theorem \ref{th1}. Then $u(t,x):=e^{\varphi(x)/h}\rho(h^2t,x)$ is a solution of \eqref{eq:heatwitten}, that is
$u(t)=e^{-t\Delta_\varphi}u_0$ with 
\begin{equation}
\begin{split}
u_0&=\rho_0e^{\varphi/2\epsilon^2}
=\left(\frac \mu{2\pi\epsilon^2}\right)^{\frac 12} \left(\sum_{n=1}^N\beta_n\indic_{E_n} +r_\epsilon\right)e^{-\varphi/2\epsilon^2}\\ 
&=\left(\frac \mu{\pi h}\right)^{\frac 12}\left( \sum_{n=1}^N\beta_n \indic_{E_n} 
+ r_h\right) e^{-\varphi/h}
\end{split}
\end{equation}
Since, $c_n(h)=(\mu/\pi)^{\frac 14}+\ooo(h)$ it follows that
$$
u_0=(\mu/\pi h)^{\frac 14} \Psi(\beta)+\tilde r_h, \ \ \ 
\tilde r_h= 
\left(\ooo(h^{\frac 12})+h^{-\frac 12}r_h  \right)e^{-\varphi/h} . $$
Since $h^{-\frac 12}e^{-\varphi/h}=\ooo_{L^1}(1)$, 
we have 
$\tilde r_h\rightarrow 0$ in $L^1$ when $h\rightarrow 0$. 
Hence, it follows from \eqref{eq:DynamWitten2} (Theorem \ref{th:DynamWitten}) that
\begin{equation*}
\begin{split}
\rho(h^2t,x)&=e^{-\varphi(x)/h}u(t,x)=e^{-\varphi(x)/h}e^{-t\Delta_\varphi}\left((\mu/\pi h)^{\frac 14} \Psi(\beta)+\tilde r_h\right)\\
&=e^{-\varphi(x)/h}\left((\mu/\pi h)^{\frac 14} \Psi(e^{-t\nu_h A_0}\beta)+e^{-t\Delta_\varphi}\tilde r_h+\ooo_{L^2}(h)\right)
\end{split}
\end{equation*}
With the new time variable $s=t \nu_h$, we obtain
\begin{equation}\label{eq:limKS1}
\rho(she^{2S/h},x)=e^{-\varphi(x)/h}\left((\mu/\pi h)^{\frac 14}\Psi(e^{-s A_0}\beta)+e^{-t\Delta_\varphi}\tilde r_h+\ooo_{L^2}(h)\right)
\end{equation}
and denoting  $\alpha(s)=e^{-s A_0}\beta$, we get
$$
e^{-\varphi(x)/h}(\mu/\pi h)^{\frac 14}\Psi(e^{-s A_0}\beta)= (\mu/\pi )^{\frac 14}\sum_{n=1}^N  \alpha_n(t)h^{-\frac 12}c_n(h)\chi_n(x)e^{-2\varphi(x)/h}.
$$
On the other hand, 
$
h^{-\frac 12}\chi_n(x)e^{-2\varphi(x)/h}\longrightarrow({\pi}/\mu)^{\frac 12}\delta_{x=m_n}
$, as $  h \to 0 $, 
in the sense of distributions.
Since, $c_n(h)=(\mu/\pi)^{\frac 14}+\ooo(h)$, it follows that
\be\label{eq:cvdistrib1}
e^{-\varphi/h}(\mu/{\pi h})^{\frac 14}\Psi(e^{-s A_0}\beta)\longrightarrow    \sum_{n=1}^N\alpha_n(t)\delta_{x=m_n}
\ee
when $h\rightarrow0$. 
Moreover, since  $e^{-t\Delta_\varphi}$ is bounded  by $1$ on $L^2$, then
$$\Vert h^{-\frac 12} e^{-\varphi/h}e^{-t\Delta_\varphi}(r_h e^{-\varphi/h})\Vert_{L^1}
\leq \Vert r_h\Vert_{L^\infty}\Vert  h^{-\frac 14} e^{-\varphi/h}\Vert_{L^2}^2\leq C\Vert r_h\Vert_{L^\infty}
$$
and recalling that $ r_h\rightarrow 0$ in $L^\infty$,
 we see that
\be\label{eq:cvdistrib2}
e^{-\varphi(x)/h}(e^{-t\Delta_\varphi}\tilde r_h+\ooo_{L^2}(h))\longrightarrow 0
\ee
in the sense of distributions.
Inserting \eqref{eq:cvdistrib1} and  \eqref{eq:cvdistrib2} into \eqref{eq:limKS1} and recalling that $h=2\epsilon^2$, we obtain \eqref{eq:weakconv}.
\end{proof}

\subsection{Witten Laplacian in dimension one} 
The Witten Laplacian is particularly simple 
in dimension one but one can already observe features which play a crucial role 
in general study.
For more information we refer to 
\cite[\S 11.1]{CyFrKiSi87_01} and \cite{He88_01}.

We first consider  $\Delta_\varphi$  acting on $ C_c^\infty (\mathbb R ) $ and recall a supersymmetric structure  which is the starting point of our analysis:
\begin{equation}\label{eq:susy1D}
\Delta_\varphi=d_\varphi^*\circ d_\varphi
\end{equation}
with $d_\varphi=e^{-\varphi/h}\circ h\partial_x\circ e^{\varphi/h}=h\partial_x+\partial_x\varphi$ and 
$d_\varphi^*=-h\partial_x+\partial_x\varphi=-d_{-\varphi}$.
From this square structure, it is clear that $\Delta_\varphi$ is non negative
and that we can use the Friedrichs extension to define a self-adjoint operator $ \Delta_\varphi $ with
domain denoted $D(\Delta_\varphi) $. Moreover, it follows from \eqref{eq:hyp-croissancephi} that there exists $c_0,h_0>0$ such that for $ 0 < h < h_0 $, 
\be\label{eq:locspecess}
\sigma_{\rm{ess}}(\Delta_\varphi)\subset [c_0,+\infty) . 
\ee
Therefore, $\sigma(\Delta_\varphi)\cap [0,c_0)$ consists of 
 eigenvalues of finite multiplicity and with no accumulation points except possibly 
$c_0$.
 
The following proposition gives a preliminary description of the low-lying eigenvalues.
\begin{proposition}\label{prop:roughestimspectre}
There exist $\varepsilon_0,h_0>0$ such that  for any $h\in (0,h_0]$, 
$\Delta_{\pm \varphi} $ has exactly $N$ eigenvalues $0 \leq \lambda^\pm_1\leq\lambda_2^\pm\ldots\leq\lambda_N^\pm$ in the interval $[0, \epsilon_0 h ]$. Moreover, 
for any $ \epsilon > 0 $ there exists $ C $ such that 
 \begin{equation}
 \lambda_n^\pm(h)\leq Ce^{-(S-\epsilon)/h},
 \end{equation}
 where $S = \sigma_1 - \varphi_0 $.
\end{proposition}

\noindent
{\bf Remark.} 
The proof applies to any $ \varphi $ which satisfies 
the first {\em two} inequalities in 
 \eqref{eq:hyp-croissancephi}. If one assumes additionally that 
$\varphi(x)\geq C\vert x\vert $ for $ |x| $ large, then $e^{-\varphi/h} \in 
D(\Delta_\varphi)$. 
 Since $d_\varphi(e^{-\varphi/h})=0$  it follows that $ \lambda_0 ^+= 0 $.

\medskip

\begin{proof}  This is proved in  \cite[Theorem 11.1]{CyFrKiSi87_01} with 
$ h^{\frac32} $ in place of $ \epsilon_0 h $. The proof applies in any 
dimension and we present it in that greater generality for $ \varphi $ satisfying
\[ \vert\partial \varphi(x)\vert\geq\frac 1 C, \ \ \ \ \vert \partial^2_{x_i x_j} \varphi\vert\leq C\vert\partial\varphi\vert^2 .\]

 The fact that there exists at least $N$ eigenvalues in the interval 
$[0,C e^{-(S- \epsilon )/h}]$ is a direct consequence of the existence of $N$ linearly independent quasi-modes  
-- see Lemma \ref{lem:propf0} and \eqref{eq:estimdeltaphifn0} below.

To show that $ N $ is the exact number of eigenvalues in 
$ [ 0 , \epsilon_0 h ) $ 
it suffices to find a $N$ dimensional vector space $V$ and $\varepsilon_0>0$ 
such that  the operator 
$\Delta_\varphi$ is bounded from below by $\varepsilon_0 h$ on $V^\bot$
-- see for instance \cite[Theorem C.15]{Zw12_01}.

To find $ V$ we introduce a family of harmonic oscillators associated to 
minima $\m\in\uuu^{(0)}$ and obtained by replacing $\varphi$ by its harmonic approximation 
in the expression for $\Delta_\varphi$:
$$
H_\m :=-h^2\Delta+\vert\varphi''(\m)(x-\m)\vert^2-h\Delta\varphi(\m), \ \ 
\m \in \mathcal U^{(0)} .
$$
The spectrum of this operator is known explicitly, see \cite[Sect 2.1]{He88_01} with the simple eigenvalue $ 0 $ at the bottom. We denote by $e_\m$ the
normalized eigenfunction, $ H_\m e_\m = 0 $. The other eigenvalues of $H_\m$ are bounded from below by $c_0h$ for some $c_0>0$.

Let $\chi\in C^\infty_c(\R^d; [0,1] )$ be equal to $ 1 $ near $ 0 $ 
and satisfy $ ( 1 - \chi^2 )^{\frac12} \in C^\infty ( \R^d ) $. 
We define $\chi_\m(x)=\chi((x-\m)/\sqrt{Mh})$ where  $M>0$  will be chosen later. For $h$ small enough, the functions $\chi_\m$ have disjoint supports and hence the function $\chi_\infty$ defined by 
$1-\chi_\infty^2=\sum_{\m\in\uuu^{(0)}}\chi_\m^2$ is smooth. We define the $N$-dimensional vector space
$$
V= {\rm{span}} \, \{ \chi_\m e_\m, \ \m\in\uuu^{(0)} \} . 
$$
The proof is completed if we show that there exist $\varepsilon_0,h_0>0$ such that 
\be\label{eq:minorDeltaphi}
\<\Delta_\varphi u,u\>\geq \varepsilon_0 h\Vert u\Vert^2,\; \ \  
\forall u\in V^\bot \cap D ( \Delta_\varphi ) , \ \ \ \forall h \in]0, h_0] .
\ee
To establish \eqref{eq:minorDeltaphi} we use 
the following localization formula  the verification of which is left to the reader 
(see \cite[Theorem 3.2]{CyFrKiSi87_01}):
$$
\Delta_\varphi=
\sum_{\m\in\uuu^{(0)}\cup\{\infty\}} \chi_\m\circ \Delta_\varphi \circ\chi_\m-h^2\sum_{\m\in\uuu^{(0)}\cup\{\infty\}} \vert\nabla\chi_\m\vert^2.
$$
Since, $\nabla\chi_\m=\ooo(({M h})^{-\frac12})$, this implies, for $ u \in D ( \Delta_\varphi ) $, that 
\be\label{eq:minorDelta0}
\<\Delta_\varphi u,u\>=\<\Delta_\varphi \chi_\infty u,\chi_\infty u\>+\sum_{\m\in\uuu^{(0)}}\<\Delta_\varphi \chi_\m u,\chi_\m u\>
+\ooo(hM^{-1}\Vert u\Vert^2).
\ee
On the support of $\chi_\infty$ we have $\vert\nabla\varphi\vert^2-h\Delta\varphi\geq(1-\ooo(h))\vert\nabla\varphi\vert^2\geq  c_1M h$ for some $c_1>0$, and hence
\be\label{eq:minorDelta1}
\<\Delta_\varphi \chi_\infty u,\chi_\infty u\>\geq Mc_1h\Vert \chi_\infty u\Vert^2
\ee
On the other hand, near any $\m\in\uuu^{(0)}$,   $\vert \nabla\varphi(x)\vert^2=\vert\varphi''(\m)(x-\m)\vert^2+\ooo(|x-\m|^3)$ 
and $\varphi''(x)=\varphi''(\m)+\ooo(|x-\m|)$. Since on the support of $\chi_\m$ we have
$ | x - \m |\leq \sqrt{Mh} $, it follows that
\begin{equation}
\label{eq:chimu}
\<\Delta_\varphi \chi_\m u,\chi_\m u\>=\<H_\m \chi_\m u,\chi_\m u\>+\ooo((Mh)^{\frac 32}).
\end{equation}
We now assume that  $u \in D ( \Delta_\varphi ) $ 
is orthogonal to $\chi_\m e_\m$ for all $\m$. Then $\chi_\m u$ is orthogonal to $e_\m$. Since the spectral gap  of $ H_\m $ is bounded from below by $ c_0 h $, \eqref{eq:chimu} shows that 
\be\label{eq:minorDelta2}
\<\Delta_\varphi \chi_\m u,\chi_\m u\>\geq c_0 h\Vert\chi_\m u\Vert^2 + \ooo((Mh)^{\frac 32}\Vert u\Vert^2) , \ \ \ \forall \m \in \mathcal U^{(0)} .
\ee
Combining this with  \eqref{eq:minorDelta0},  \eqref{eq:minorDelta1} and \eqref{eq:minorDelta2} gives
\begin{equation*}
\begin{split}
\<\Delta_\varphi u,u\>&\geq c_0 h\sum_{\m\in\uuu^{(0)}\cup\{\infty\}}\Vert\chi_\m u\Vert^2+\ooo(hM^{-1}\Vert u\Vert^2)+\ooo((Mh)^{\frac 32}\Vert u\Vert^2)\\
&\geq  c_0h\Vert u\Vert^2+\ooo(hM^{-1}\Vert u\Vert^2)+\ooo((Mh)^{\frac 32}\Vert u\Vert^2).
\end{split}
\end{equation*}
Taking $M$ large enough completes the proof of \eqref{eq:minorDeltaphi}.\end{proof}

We denote by $E^{(0)}$ the
subspace spanned by eigenfunctions of these low lying eigenvalues and  by
 \begin{equation}\label{eq:definPi0}
 \Pi^{(0)} :=\indic_{[0,\varepsilon_0h ]}(\Delta_\varphi)
 \end{equation}
the spectral projection onto $E^{(0)}$. This projector is 
expressed by the standard contour integral
\begin{equation} \label{eq:cauchyPi0}
\Pi^{(0)} = \frac{1}{2\pi i} \int_{\partial B ( 0 , \delta \varepsilon_0h ) }
(z-\Delta_\varphi)^{-1} d z 
\end{equation}

In our analysis, we will also need the operator $\Delta_{-\varphi}$, noting  
that in dimension one $\Delta_{-\varphi}$ is the Witten Laplacian on $1$-forms. 
Since $-\varphi$ has exactly $N-1$ minima (given by the $N-1$ maxima of $\varphi$), it follows from Proposition \ref{prop:roughestimspectre} that there exists $\varepsilon_1>0$ such that 
$\Delta_{-\varphi}$ has $N-1$ eigenvalues in $[0,\varepsilon_1 h]$ and that these eigenvalues 
are actually exponentially small. Observe that because of the condition $\varphi(x)\geq C\vert x\vert$ at infinity, the function $e^{\varphi/h}$ is not square integrable. Consequently, unlike in the case of $ \Delta_{\varphi } $, we cannot conclude that 
the lowest eigenvalue is equal to $0$

We denote by $E^{(1)}$ the
subspace spanned by eigenfunctions of these low-lying eigenfunctions of $\Delta_{-\varphi}$ and  by
$\Pi^{(1)}$ the corresponding projector onto $E^{(1)}$, 
\begin{equation}\label{eq:definPi1}
 \Pi^{(1)}=\indic_{[0,\varepsilon_1h]}(\Delta_{-\varphi}).
 \end{equation}
Similarly to \eqref{eq:cauchyPi0}, we have
 \begin{equation} \label{eq:cauchyPi1}
\Pi^{(1)} = \frac{1}{2\pi i} \int_{\partial B ( 0 , \delta \varepsilon_1 h ) } (z-\Delta_{-\varphi})^{-1} d z , 
\end{equation}
for any  $0 < \delta <1$.

\subsection{Supersymmetry}  The key point in the analysis is the following intertwining relations which follows directly from \eqref{eq:susy1D}
\begin{equation}\label{eq:intertwin}
\Delta_{-\varphi} \circ d_\varphi=d_\varphi\circ \Delta_\varphi
\end{equation}
and its adjoint relation
\begin{equation}\label{eq:intertwinadj}
d_\varphi^*\circ\Delta_{-\varphi}= \Delta_\varphi\circ d_\varphi^*.
\end{equation}
 From these relations we deduce that $d_\varphi(E^{(0)})\subset E^{(1)} $ and $d_\varphi^*(E^{(1)})\subset E^{(0)} $. Indeed, suppose that 
 $\Delta_\varphi u=\lambda u$, 
 with $u\neq 0$ and $\lambda\in[0,\varepsilon_0 h]$. Then, we see from \eqref{eq:intertwin} that
 $$
 \Delta_{-\varphi}(d_\varphi u)=d_\varphi( \Delta_\varphi u)=\lambda d_\varphi u.
 $$
 Therefore, either $d_\varphi u$ is null and obviously belongs to $E^{(1)}$ or $d_\varphi u\neq 0$ and hence $d_\varphi u$ is an eigenvector of $\Delta_{-\varphi}$ associated to $\lambda\in[0,\varepsilon_0 h]$. This proves the first statement. The inclusion $d_\varphi^*(E^{(1)})\subset E^{(0)} $ is obtained by similar arguments.
 
 By definition, the operator $\Delta_\varphi$ maps $E^{(0)}$ into itself and we can consider its restriction to $E^{(0)}$.  From the above discussion we know also that $d_\varphi(E^{(0)})\subset E^{(1)} $ and $d_\varphi^*(E^{(1)})\subset E^{(0)} $. Hence we consider
$ \lll=(d_\varphi)_{\vert E^{(0)}\rightarrow E^{(1)}} $ and $  \lll^*=(d_\varphi^*)_{\vert E^{(1)}\rightarrow E^{(0)}}$.
When restricted to $E^{(0)}$,  the structure equation \eqref{eq:susy1D} becomes
 \begin{equation}\label{eq:susymatrice}
 \mmm=\lll^*\lll \ \ \ \text{ with }\ \ \ \ \ \mmm := \Delta_\varphi |_{ E^{(0)} } , \ \ \
 \lll := (d_\varphi)|_{ E^{(0)} \to E^{(1)}} . 
 \end{equation}
 
 \subsection{Quasi-modes for $\Delta_\varphi$}
Let $\delta_0=\inf\{\diam(E_n),\,n=1,\ldots,N\}$ and let $\epsilon>0$ be small with respect to $\delta_0$. For all $n=1,\ldots,N$, let $\chi_n$ be smooth cut-off functions such that 
\begin{equation}
\label{eq:cutoff} \left\{
\begin{array}{c}
0\leq \chi_n\leq 1,\\
\supp(\chi_n)\subset \{x\in E_n,\,\varphi(x)\leq\sigma_1- \epsilon \}\\
\chi_n=1\text{ on } \{x\in E_n,\,\varphi(x)\leq \sigma_1- 2\epsilon\},
\end{array}
\right. 
\end{equation}
where $ \epsilon > 0$ will be chosen small (in particular much smaller than 
$ \delta_0 $ in \eqref{eq:delta0}).
Consider now the family of approximated eigenfunctions defined by 
\begin{equation}\label{eq:definfn0}
f_n^{(0)}(x)=h^{-\frac 14}c_n(h)\chi_n(x) e^{-\varphi(x)/h} , \ \ 
\ \Vert f_n^{(0)}\Vert_{L^2}=1 , 
\end{equation}
where $c_n(h)=\varphi''(m_n)^{\frac 14}\pi^{-\frac 14}+\ooo(h)$. 
 We introduce the projection of these quasi-modes onto the eigenspace space $E^{(0)}$:
\begin{equation}\label{eq:defingn0}
g_n^{(0)}:=\Pi^{(0)}f_n^{(0)}.
\end{equation}
\begin{lemma}\label{lem:propf0}
The approximate eigenfunctions defined by \eqref{eq:definfn0}
 satisfy 
$$ \<f_n^{(0)},f_m^{(0)}\>=\delta_{n,m}, \ \ \   \forall n,m=1,\ldots ,N,
$$
and 
$$d_\varphi f_n^{(0)}=\ooo_{L^2}(e^{-(S-\epsilon)/h}), \ \ \ \ 
g_n^{(0)}-f_n^{(0)}=\ooo_{L^2}(e^{-(S-\epsilon')/h})$$
for any $\epsilon'>\epsilon$.
\end{lemma}
\begin{proof}
The first statement is a direct consequence of the support properties of the cut-off functions $\chi_n$ and the choice of the normalizing constant.
To see the second estimate, recall that $d_\varphi e^{-\varphi/h}=0$. Hence
$$
d_\varphi f_n^{(0)}(x)=h^{\frac 3 4} c_n(h)\chi_n'(x) e^{-\varphi(x)/h}.
$$
Moreover, thanks to \eqref{eq:cutoff}, there exists $c>0$ such that 
for $\epsilon>0$ small enough we have
$ \varphi(x)\geq S-\epsilon $ {for} $ x\in \supp(\chi_n')$. 
Combining these two facts gives estimates on $d_\varphi f_n^{(0)} $. 

We now prove the  estimate on  $ g_n^{(0)}-f_n^{(0)}$. We first observe that 
$$
\Delta_\varphi f_n^{(0)}=d_\varphi^* d_\varphi f_n^{(0)}=h^{\frac 3 4} c_n(h)d_\varphi^*(\chi_n' e^{-\varphi/h})
=h^{\frac 3 4} c_n(h)(-h\chi_n''+2\partial_x\varphi \chi'_n) e^{-\varphi/h}
$$
and the same argument as before shows that 
\begin{equation}\label{eq:estimdeltaphifn0}
\Delta_\varphi f_n^{(0)}=\ooo_{L^2}(e^{-(S-\epsilon)/h}).
\end{equation}
From
\eqref{eq:cauchyPi0} and  Cauchy formula, it follows that
\begin{align*}
g_n^{(0)} - f_n^{(0)}
& = \Pi^{(0)} f_n^{(0)} - f_n^{(0)} 
= \frac{1}{2 \pi i } \int_\gamma (z - \Delta_\varphi)^{-1} f_n^{(0)} dz - \frac{1}{2 \pi i } \int_\gamma z^{-1} f_n^{(0)}d z \\
& = \frac{1}{2 \pi i } \int_\gamma (z -  \Delta_\varphi)^{-1} z^{-1}  \Delta_\varphi f_n^{(0)} d z ,  
\end{align*}
with $\gamma = \partial B ( 0 , \delta\epsilon_0 h ), 0 < \delta < 1 $.
Since $\Delta_\varphi$ is selfadjoint and $\sigma(\Delta_\varphi)\cap[0,\epsilon_0 h]\subset [0,e^{-1/Ch}]$, we have for $\alpha>0$ small enough
\begin{equation*}
\big\Vert ( z - \Delta_\varphi )^{- 1} \big\Vert = \ooo ( h^{-1} ) ,
\end{equation*}
uniformly for $z \in \gamma$. Using \eqref{eq:estimdeltaphifn0}, we get
$\big\Vert ( z - \Delta_\varphi)^{- 1} z^{- 1} \Delta_\varphi f_k^{(0)} \big\Vert = \ooo \big( h^{- 2} e^{-(S-\epsilon) / h} \big) $,
and, after integration,
$ \| g_k^{(0)} - f_k^{(0)} \|  = \ooo ( h^{-1 } e^{- (S-\epsilon) / h} ) = \ooo (  e^{- (S-\epsilon') / h} ), $
for any $\epsilon'>\epsilon$.
\end{proof}

 \subsection{Quasi-modes for $\Delta_{-\varphi}$} Since, $\varphi$ and $-\varphi$ share similar properties, the construction of the preceding section produces quasi-modes for $\Delta_{-\varphi}$. Eventually we will only need  quasi-modes localized near the maxima $\s_k$. Hence, let 
  $\theta_k \in C^\infty_c ( \R ; [ 0 , 1] ) $ satisfy
 \begin{equation}
 \label{eq:delta0} \supp\theta_k\subset \{\vert x-\s_k\vert\leq\delta_0\}, \ \ 
 \text{ $\theta_k=1$ on $\{\vert x-\s_k\vert\leq \frac{\delta_0}2\}$. } 
 \end{equation}
We take $ \epsilon $ in the definition \eqref{eq:cutoff} small enough
then for all $k=1,\ldots N-1$, we have
 \begin{equation}\label{eq:propsupptronc}
 \theta_k\chi'_k=\chi'_{k,+}\text{ and } \theta_k\chi'_{k+1}=\chi'_{k+1,-}
 \end{equation}
 where $\chi_{k,\pm}$ are the smooth functions defined by  
 $$
 \chi_{k,+}(x)=\left\{
 \begin{array}{cc}
 \chi_k(x)&\text{ if }x\geq m_k,\\
 1&\text{ if }x<m_k,
 \end{array}
 \right.
\ \ \ \
 \chi_{k,-}(x)=\left\{
 \begin{array}{cc}
 \chi_k(x)&\text{ if }x\leq m_k,\\
 1&\text{ if }x>m_k.
 \end{array}
 \right.
 $$
Moreover, we also have $\theta_k\theta_l=0$ for all $k\neq l$. The family of quasi-modes associated to these cut-off functions is given by
\begin{equation}\label{eq:definfn1}
f_k^{(1)}(x):=h^{-\frac 14}d_k(h)\theta_k(x) e^{(\varphi(x)-S)/h}, \ \ \
\Vert f_k^{(1)}\Vert_{L^2}=1 , 
\end{equation}
where $d_k(h)=\vert\varphi''(s_k)\vert^{\frac 14}\pi^{-\frac 14}+\ooo(h)$ is the 
normalizing constant. Again,  we introduce the projection of these quasi-modes onto the eigenspace $E^{(1)}$:
\begin{equation}\label{eq:definen1}
g_k^{(1)}(x):=\Pi^{(1)}f_k^{(1)}.
\end{equation}
\begin{lemma}\label{lem:propf1}
There exists $\alpha>0$ independant of $\epsilon$ such that the following hold true
$$
\<f_k^{(1)},f_l^{(1)}\>=\delta_{k,l}, \ \ \ \forall  k,l=1,\ldots ,N-1,
$$
$$d_\varphi^* f_k^{(1)}=\ooo_{L^2}(e^{-\alpha/h}), \ \ \ g_k^{(1)}-f_k^{(1)}=\ooo_{L^2}(e^{-\alpha/h})$$
\end{lemma}
\begin{proof}
The proof follows the same lines as the proof of Lemma \ref{lem:propf0}.
\end{proof}

\subsection{ Computation of the operator $\lll$}
 In this section we represent $\lll$ in a suitable basis.
 For that we first observe that the bases $(g_n^{(0)})$ and $(g_k^{(1)})$ are quasi-orthonormal. Indeed, thanks to  Lemmas \ref{lem:propf0} and \ref{lem:propf1}, we have 
$$
\<g_n^{(0)},g_m^{(0)}\>=\delta_{n,m}+\ooo(e^{-\alpha/h}), \ \ \  \forall n,m=1,\ldots ,N
$$
and 
$$
\<g_k^{(1)},g_l^{(1)}\>=\delta_{k,l}+\ooo(e^{-\alpha/h}), \ \ \   \forall k,l=1,\ldots ,N-1.
$$
for some $\alpha>0$.
We then obtain orthonormal bases of $ E^{(0)} $ and $ E^{(1)} $:
\begin{equation}
\label{eq:GrSch}
\begin{split} 
(g_n^{(0)} )_{ 1\leq n \leq N } & \xrightarrow{\text{Gramm--Schmidt process}} (e_n^{(0)})_{ 1 \leq n \leq N } , \\ 
(g_k^{(1)} )_{ 1\leq k \leq N -1 } & \xrightarrow{\text{Gramm--Schmidt process}} (e_k^{(1)})_{ 1 \leq n \leq N -1 } .
\end{split} 
\end{equation} 
It follows from the approximate orthonormality above 
that the change of basis matrix $P_j$ from $(g_n^{(j)})$ to
$( e_n^{(j)})$ satisfies
\begin{equation}\label{eq:passage}
P_j=\Id +\ooo(e^{-\alpha/h})
\end{equation}
for $j=0,1$.
To describe the matrix of $\lll$ in the bases $(e_n^{(0)})$ and $(e_k^{(1)})$ we introduce 
a $N-1\times N$ matrix $\hat L=(\hat\ell_{ij})$ defined by 
 \begin{equation}
 \label{eq:lhat}
 \hat\ell_{ij}=\<f_i^{(1)},d_\varphi f_j^{(0)}\>.
 \end{equation}
We claim that the matrices $L$ and $\hat L$ are very close. To see
that we give a precise expansion of $\hat L$:
 \begin{lemma}\label{lem:computhatL}
 The matrix $ \hat L $ defined by \eqref{eq:lhat} is given by  
   $\hat L=( h/\pi)^{\frac 12}e^{-S/h} \bar L $ where 
 $\bar  L$ admits a classical expansion $\bar L\sim\Sigma_{k=0}^\infty h^k  L_k$ with 
 \begin{equation}
 \label{eq:matrixL}
 L_0= 
 \left(
 \begin{array}{ccccccc}
-\nu_1^{\frac14} \mu_1^{\frac14} & \ \nu_1^{\frac14} \mu_2^{\frac14} &0&0&\ldots&0\\
0&- \nu_2^{\frac14} \mu_2^{\frac14} & \nu_2^{\frac14} \mu_3^{\frac14} &0&\ldots&0\\
\vdots&\ddots&\ddots&\ddots&\ddots&\vdots\\
\vdots&\ddots&\ddots&\ddots&\ddots&0\\
0&0&\ldots&0&  - \nu_{n-1}^{\frac14} \mu_{n-1}^{\frac14} & 
 \nu_{n-1}^{\frac14} \mu_{n}^{\frac14} 
\end{array}
 \right).
\end{equation}
 \end{lemma}
\begin{proof}
From \eqref{eq:definfn0} and  \eqref{eq:definfn1}, we have
\begin{equation*}
\begin{split}
 \hat\ell_{ij}&=\<f_i^{(1)},d_\varphi f_j^{(0)}\>
 =h^{-\frac 12}d_i(h)c_j(h)\int_\R \theta_i(x) e^{(\varphi(x)-S)/h}d_\varphi(\chi_j(x)e^{-\varphi(x)/h})dx\\
 &=h^{\frac 12}d_i(h)c_j(h)e^{-S/h}\int_\R\theta_i(x)\chi_j'(x)dx.
\end{split}
\end{equation*}
Moreover, since $\supp \theta_i\cap \supp \chi_j=\emptyset $ except for $j=i$ or $j=i+1$, it follows from \eqref{eq:propsupptronc} that 
\begin{equation}
\label{eq:psichi}
\int_\R\theta_i(x)\chi_j'(x)dx=\delta_{i,j}\int_\R\chi_{i,+}'(x)dx+\delta_{i+1,j}\int_\R\chi_{i,-}'(x)dx=-\delta_{i,j}+\delta_{i+1,j} .
\end{equation}
On the other hand, we recall that $d_i(h)$ and $c_j(h)$ both have a classical expansion. Together with the above equality, this shows that 
$\hat L$ has the required form and it remains to prove the formula giving $L_0$.
To that end we observe that 
$$
d_i(h)c_j(h)=\pi^{-\frac 12}((\vert\varphi''(s_i)\vert\varphi''(m_j))^{\frac 14}+\ooo(h))=\mu_j^{\frac 14}\nu_i^{\frac14}\pi^{-\frac 12} +\ooo(h)
$$
in the notation of \eqref{eq:nothessphi}. Combining this with \eqref{eq:psichi}
we obtain 
$$
 \hat\ell_{ij}= h^{\frac 12}\pi^{-\frac 12} e^{-S/h} \mu_j^{\frac 14}\nu_i^{\frac14} (-\delta_{i,j}+\delta_{i+1,j}+\ooo(h))
$$
which gives \eqref{eq:matrixL}.
\end{proof}

 \begin{lemma}\label{lem:approxL} Let $ L $ be the matrix 
 of $ \mathcal L $ in the basis obtained in \eqref{eq:GrSch}. 
 There exists $\alpha'>0$ such that 
 $ L=\hat L+\ooo(e^{-(S+\alpha')/h})$, where $ \hat L $ is defined by \eqref{eq:lhat} and is described in Lemma \ref{lem:computhatL}.
 \end{lemma}
 \begin{proof}
  It follows from \eqref{eq:passage} that
\be\label{eq:chbase1}
 L=(\Id+\ooo(e^{-\alpha/h})) \tilde L(\Id+\ooo(e^{-\alpha/h}))
\ee
where $\tilde L=(\tilde \ell_{i,j})$ with $\tilde \ell_{i,j}=\<g_i^{(1)},d_\varphi g_j^{(0)}\>$. Moreover, \eqref{eq:intertwin} implies that 
$\Pi^{(1)}d_\varphi=d_\varphi \Pi^{(0)}$. Using this identity and the fact that $\Pi^{(0)},\Pi^{(1)}$ are orthogonal projections, we have
\begin{equation*}
\begin{split}
\<g^{(1)}_i,d_\varphi g^{(0)}_j\>&=\<g^{(1)}_i,d_\varphi \Pi^{(0)}f^{(0)}_j\>=\<g^{(1)}_i,\Pi^{(1)}d_\varphi f^{(0)}_j\>=\<g^{(1)}_i,d_\varphi f^{(0)}_j\>\\
&=\<f^{(1)}_i,d_\varphi f^{(0)}_j\>+\<g^{(1)}_i-f^{(1)}_i,d_\varphi f^{(0)}_j\>
\end{split}
\end{equation*}
But from Lemmas \ref{lem:propf0}, \ref{lem:propf1} and the Cauchy-Schwarz inequality we get
$$
\vert \<g^{(1)}_i-f^{(1)}_i,d_\varphi f^{(0)}_j\>\vert\leq C e^{-(\alpha+S-\epsilon')/h}.
$$
Since $\alpha$ is independent of $\epsilon'$ which can be chosen as small as we want, it follows that there exists $\alpha'>0$ such that 
$\tilde\ell_{i,j}=\hat \ell_{i,j}+\ooo(e^{-(S+\alpha')/h})$. Combining this estimate, \eqref{eq:chbase1} and the fact that $ \hat \ell_{i,j}=\ooo(e^{-S/h})$, we get 
the announced result.
\end{proof}

It is now easy to describe 
$\mmm$ as a matrix:
 \begin{lemma}\label{lem:asymptM}
Let $M$ be the matrix representation of $\mmm$ in the basis $(e^{(0)}_n)$. Then 
$$M= h e^{-2S/h}A$$
 where  $A$ is symmetric positive with a classical expansion $A\sim\sum_{k=0}^\infty h^k A_k$
with $A_0$ given by \eqref{eq:defA0}.
\end{lemma}
\bp
By definition,  $M=L^*L$ and it follows from Lemma \ref{lem:computhatL} and  \ref{lem:approxL} that 
$$
L^*L=(\hat L+\ooo(e^{-(S+\alpha')/h}))^*(\hat L+\ooo(e^{-(S+\alpha')/h}))
= h e^{-2S/h}(\bar L^*\bar L+\ooo(e^{-\alpha'/h}))
$$
Then, $A:=h^{-1}e^{2S/h}L^*L$ is clearly positive  and admits a classical expansion since $\bar L$ does. Moreover, the leading term of this expansion is 
$\bar L_0^*\bar L_0$ and 
 a simple computation shows that $\bar L_0^*\bar L_0=A_0$, where
 $ A_0 $ is given by \eqref{eq:defA0}.
\ep

\noindent
{\bf Remark.} Innocent as this lemma might seem, the supersymmetric structure,
that is writing $ - \Delta_\varphi |_{ E^{(0)}} $ using $ d_\varphi $, is very 
useful here. 

\begin{lemma}\label{cor:spectM}
Denote by $\mu_1(h)\leq\ldots\leq \mu_k(h)$ the eigenvalues of $A(h)$. Then, 
$$
\mu_0(h)=0\text{ and }\mu_k(h)=\mu^0_k+\ooo(h), \ \ \ \forall k\geq 2,
$$
where $0=\mu^0_1<\mu^0_2\leq\mu^0_3\leq\ldots\leq\mu^0_N$ denote the eigenvalues of $A_0$.
Moreover, a normalized
eigenvector associated to $\mu_1^0$ is $\xi^0=N^{-\frac 12}(1,\ldots,1)$ and there exists a normalized vector $\xi(h)\in\ker(A(h))$, such that 
\be\label{eq:comparnoyau}
\xi(h)=\xi_0+\ooo(h).
\ee
\end{lemma}
\bp
Many of the statements of this lemma are immediate consequence of Lemma \ref{lem:asymptM}. We emphasize the fact that
$0$ belongs to $\sigma(A)$ since $0\in\sigma(\Delta_\varphi)$. 
The fact that $\xi^0$ is in the kernel of $A_0$ is a simple computation. Eventually, for any $\xi\in \ker(A(h))$,  we have
$$
\xi-\<\xi,\xi_0\>\xi_0=\frac 1{2i\pi}\Big(\int_{\gamma}z^{-1}\xi dz-\int_{\gamma}(A_0-z)^{-1}\xi dz\Big) =\frac 1{2i\pi}\int_{\gamma}(A_0-z)^{-1}z^{-1}A_0\xi dz
$$
where $\gamma$ is a small path around $0$ in $\C$. Since $A_0\xi=\ooo(h)$ we obtain \eqref{eq:comparnoyau}.
\ep

 \subsection{Proof of Theorem \ref{th:DynamWitten}}
 \label{s:pft}
 Let $u$ be solution of \eqref{eq:heatwitten} with $u_0=\Psi(\beta)$, 
 $ | \beta | \leq 1 $ (see
 \eqref{eq:defpsin},\eqref{eq:defPsi} and \eqref{eq:definfn0} for 
 definitions of $ \psi_n$, $ \Psi$ and $ f_n^{(0)} $, respectively). Then,
\begin{equation}
\begin{split}
 u&=e^{-t\Delta_\varphi}\Pi^{(0)}u_0+e^{-t\Delta_\varphi}\widehat\Pi^{(0)}u_0\\
 &=e^{-t\mmm}\Pi^{(0)}u_0+e^{-t\Delta_\varphi}\widehat\Pi^{(0)}u_0, \ \ \
 \widehat\Pi^{(0)}: = I-\Pi^{(0)}.
 \end{split}
\end{equation}
Since $\indic_{E_n}-\chi_n$ is supported near $\{s_{n-1},s_n\}$, then 
 $\psi_n-f_n^{(0)}=\ooo_{L^2}(e^{-\alpha/h})$, for all $n=1,\ldots, N$, and it follows that
 $$
 u_0=\bar u_0+\ooo_{L^2}(e^{-\alpha/h}), \ \ \ 
 \bar u_0 :=\sum_{n=1}^N\beta_n f_n^{(0)}.
 $$
 Then, using Lemma \ref{lem:propf0} and \eqref{eq:passage}, we get 
 $
 u_0=\tilde u_0+\ooo_{L^2}(e^{-\alpha/h})
 $
 with 
 $
 \tilde u_0:=\sum_{n=1}^N\beta_n e_n^{(0)}, 
 $
 where $e_n^{(0)}$ is the orthonormal basis of $E^{(0)}$ given by \eqref{eq:GrSch}. Since 
 $\Pi^{(0)}e_n^{(0)}=e^{(0)}_n$ and $\widehat\Pi^{(0)}e_n^{(0)}=0$, we have
 $$
 u(t)=e^{-t\mmm}\tilde u_0+\ooo_{L^2}(e^{-\alpha/h}).
 $$
 If $M$ is the matrix of the operator $\mmm$ in the basis $(e^{(0)}_n)$ then
 $$
 u(t)=\sum_{n=1}^N(e^{-tM}\beta)_n e_n^{(0)}+\ooo_{L^2}(e^{-\alpha/h}).
 $$
 Going back from $e_n^{(0)}$ to $\psi_n$ as above, we see that
 \begin{equation}
 \label{eq:utpsi}
 u(t)=\sum_{n=1}^N(e^{-tM}\beta)_n \psi_n+\ooo_{L^2}(e^{-\alpha/h})=\Psi(e^{-tM}\beta)+\ooo_{L^2}(e^{-\alpha/h})
 \end{equation}
 and the proof of \eqref{eq:DynamWitten1} (main statement in Theorem  \ref{th:DynamWitten}) is complete. We now prove \eqref{eq:DynamWitten2}. Since the linear map $\psi:\C^N\rightarrow L^2(dx)$ 
is bounded uniformly with respect to $h$, and thanks to \eqref{eq:DynamWitten1}, the proof reduces  to showing (after time rescaling) that there exists $C>0$ such that for all $\beta\in\R^N$,
 \be\label{eq:DynamWitten3}
 \vert e^{-\tau A}-e^{-\tau A_0}\vert\leq Ch,\; \ \  \forall \tau\geq 0
 \ee
Since, by Lemma \ref{cor:spectM}, $ A $ and $ A_0 $ both have $ 0 $ as a simple eigenvalue with the approximate eigenvector given by $ ( 1, \cdots , 1 ) $,  
we see that for any norm on $ \C^N $, 
\[ \begin{split} | e^{ - \tau A } - e^{ - \tau A_0 } |
& \leq | e^{ - \tau A_0 } |_{ \{  ( 1, \cdots , 1 ) \}^\perp} 
  | I - e^{ - \tau \mathcal O ( h ) } |_{ \ell^2 \to \ell^2 } + 
 C h  \\
&   \leq C e^{ - c \tau } \tau h + C h  = 
  \mathcal O ( h ) . \end{split} \]
which is exactly \eqref{eq:DynamWitten3}.
\ep

We now prove one of the consequences of Theorem  \ref{th:DynamWitten}.
 \begin{proof}[Proof of \eqref{eq:rhot}]
We have seen in the preceding proof that  $e_n^{(0)}-\psi_n=\ooo_{L^2}(e^{-C/\epsilon^2})$ and since 
\[  \| \psi_n - ( \mu/ 2 \pi \epsilon^2 )^{\frac14} \indic_{E_n }e^{-\varphi/2\epsilon^2} \|_{L^2} = 
\mathcal O ( \epsilon^2 ) ,\]
it follows that 
$\Pi^{(0)}u_0=\psi(\beta)+\ooo(\epsilon^2\Vert u_0\Vert_{L^2})$ with $\beta\in\C^N$ given by
$$
\beta_n = (\tfrac{ \mu}{ 2 \pi \epsilon^2 } )^{\frac14} \int_{ E_n}  u_0 ( x ) 
e^{ -\varphi(x) /2\epsilon^2 } dx= (\tfrac{ \mu}{ 2 \pi \epsilon^2 } )^{\frac14} \int_{ E_n}  \rho_0 ( x ) 
 dx . 
$$
Applying  \ref{th:DynamWitten} (second part of Theorem \ref{th:DynamWitten}) with 
 $h=2\epsilon^2$ gives
\be\label{equa1}
 e^{- t \Delta_\varphi } \Pi^{(0)} u_0 = \sum_{ n=1}^N ( e^{ - t \nu_h A_0 } 
\beta )_n (\tfrac{ \mu}{ 2 \pi \epsilon^2 } )^{\frac14} \indic_{E_n } e^{ -\varphi/2\epsilon^2}  + 
\mathcal O_{L^2} ( \epsilon^2 ) \| u_0\|_{L^2}.
\ee
 On the other hand, Proposition \ref{prop:roughestimspectre} shows that
\be\label{equa2}
e^{-t \Delta_\varphi } (I-\Pi^{(0)} ) u_0 = \mathcal O_{L^2} ( 
e^{ - t \epsilon^2 /C } ) \| u_0 \|_{ L^2 } .
\ee
Since 
$ \rho ( h^2 t ) = e^{-\varphi/h} u ( t ) $, \eqref{equa1} and \eqref{equa2} yield 
\begin{equation}
\label{eq:rhotC} \begin{gathered} \rho ( 2 \epsilon^2 e^{ S/\epsilon^2 } \tau ) = 
\sum_{ n=1}^N ( e^{ - \tau A_0 } 
\beta )_n (\tfrac{ \mu}{ 2 \pi \epsilon^2 } )^{\frac14} \indic_{E_n }e^{ -\varphi/\epsilon^2} + 
r_\epsilon(\tau)
\end{gathered}
\end{equation}
with
$$
r_\epsilon(\tau)=e^{-\varphi/2\epsilon^2}\Big(\ooo_{L^2}(e^{-c \tau e^{ S/ \epsilon^2}})+\ooo_{L^2}(\epsilon^2)\Big)\| \rho_0 \|_{L^2_\varphi }.
$$
By Cauchy-Schwartz it follows that 
$\Vert r_\epsilon(\tau)\Vert_{L^1}\leq C( \epsilon^{\frac 52}  + e^{ - c \tau e^{ S/ \epsilon^2}}
) \| \rho_0 \|_{L^2_\varphi } $.
 \end{proof}

 \section{A higher dimensional example}
 \label{s:get_high}
 
The same principles apply when the wells may have different height and in 
higher dimensions. In both cases there are interesting 
combinatorial and topological (when $ d > 1 $) complications and we refer to 
\cite[\S 1.1, \S 1.2]{Mi16} for a presentation and references. 
To illustrate this we give
 a higher dimensional result in a simplified setting.

Suppose that $\varphi:\R^d\rightarrow\R$ is a smooth Morse function 
satisfying \eqref{eq:hyp-croissancephi} and denote  by $\uuu^{(j)}$ the finite sets of critical points of index $j$,  $n_j := |\uuu^{(j)}| $. We assume that 
\eqref{eq:hyptopol} holds and write $ S := \sigma_1 - \varphi_0 $.
In the notation of \eqref{eq:En} we have $ n_0 = N $ and we also assume that each $ E_n $ contains
exactly one minimum. Hence we can label the components by the minima:
\[   \forall n=1,\ldots, N,\;\;\exists \, ! \, \m\in E_n ,\;\min_{  x \in E_n } \varphi ( x ) = \varphi ( \mathbf m )\]
 and  we denote $ E( \mathbf m ) := E_n $.
Since $\varphi$ is a Morse function, 
\begin{equation}
\label{eq:grpr} 
\begin{gathered} \forall  \m,\m'\in \uuu^{(0)}, \ \ \m \neq \m' \ \Longrightarrow 
\bar E(\m)\cap\bar E(\m') \subset \uuu^{(1)}, \\
 \forall \, \s\in\uuu^{(1)}, \ \exists\, ! \, \m , \m' \in 
\mathcal U^{(0)} , \ \ \ \ \s\in\bar E(\m)\cap\bar E(\m') . 
\end{gathered}
\end{equation}
To simplify the presentation we make an addition assumption
\begin{equation}
\label{eq:h3} 
\forall \m,\m'\in \uuu^{(0)}, \ \ \m \neq \m' \ \Longrightarrow 
|\bar E(\m)\cap\bar E(\m') | \leq 1 . 
\end{equation}
Under these assumptions, the set $\uuu^{(0)}\times \uuu^{(1)}$ defines a graph $\ggg$. The elements of $\uuu^{(0)}$ are the vertices of $ \ggg $ and  elements of 
$\uuu^{(1)}$ are the edges of $ \ggg $: $\s\in\uuu^{(1)}$ is  an edge between 
$ \m $ and $ \m' $ in $ \uuu^{(0)} $
if $ s \in\bar E ( \m ) \cap E ( \m') $
 -- see Fig.\ref{figsublevel} for an example.

The same graph has been constructed in \cite{Lamit} for a certain discrete
model of the Kramers--Smoluchowski equation.

We now introduce 
the discrete Laplace operator on $ \ggg $, $M _\ggg $ -- see 
\cite{CvDoSa95} for the background and results about $ M_\ggg $.
If 
the degree ${d}(\m)$ is defined as the number of edges 
at the vertex $\m$,  $M _\ggg $ is given by the matrix $ (a_{\m,\m'})_{\m,\m'\in\uuu^{(0)}}$:
\begin{equation}\label{eq:deflaplacegraphe}
a_{\m,\m'}=
\left\{
\begin{array}{ll}
{d}(\m), & \m=\m' \\
-1  &  \m\neq \m', \ \ \bar E(\m)\cap\bar E(\m') \neq \emptyset , \\
\ \ \, 0  & \text{ otherwise}
\end{array}
\right.
\end{equation}

\begin{figure}
 \center
  \scalebox{0.35}{ 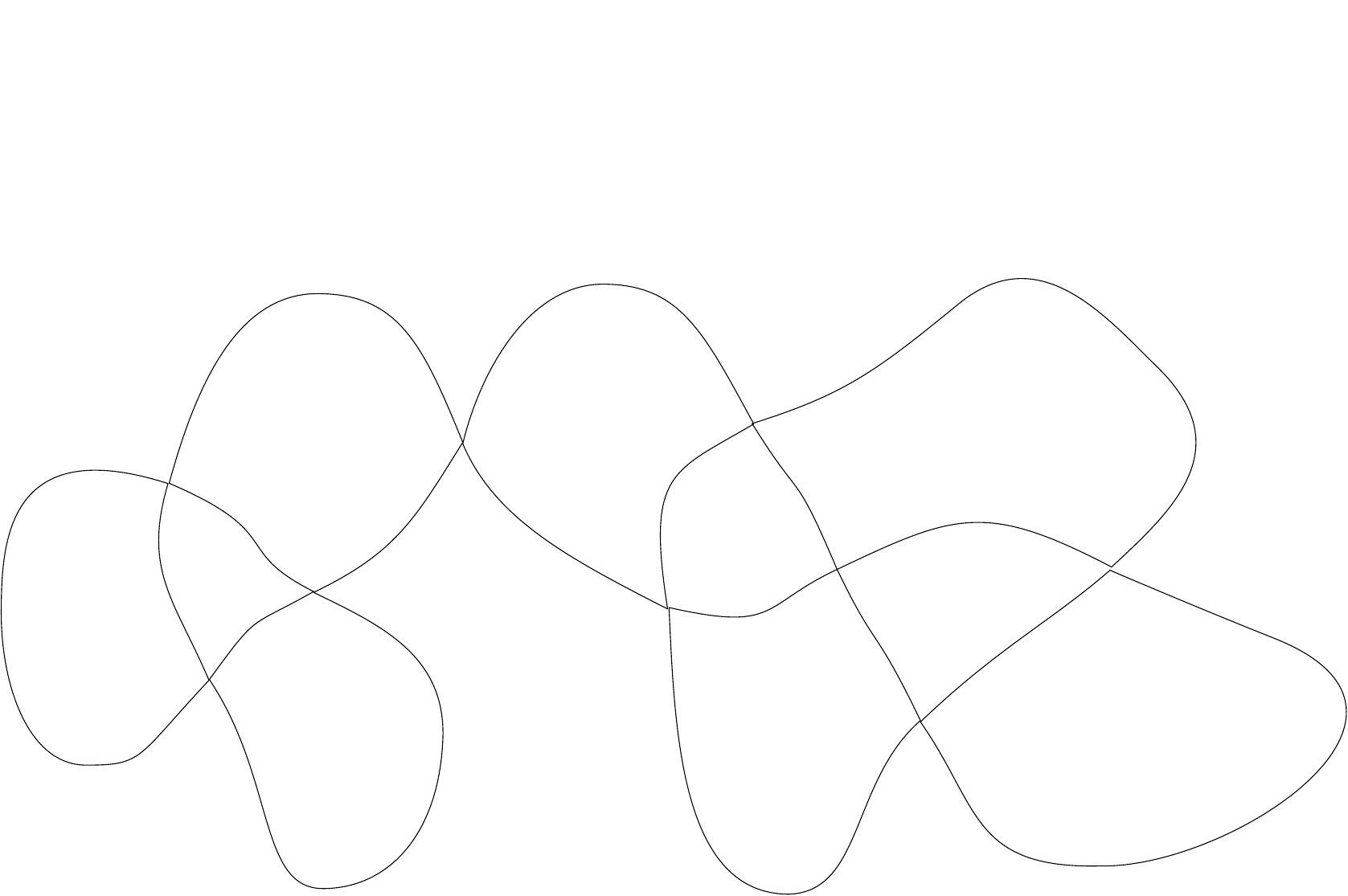 } \hspace{0.1in}
\scalebox{0.6}{ 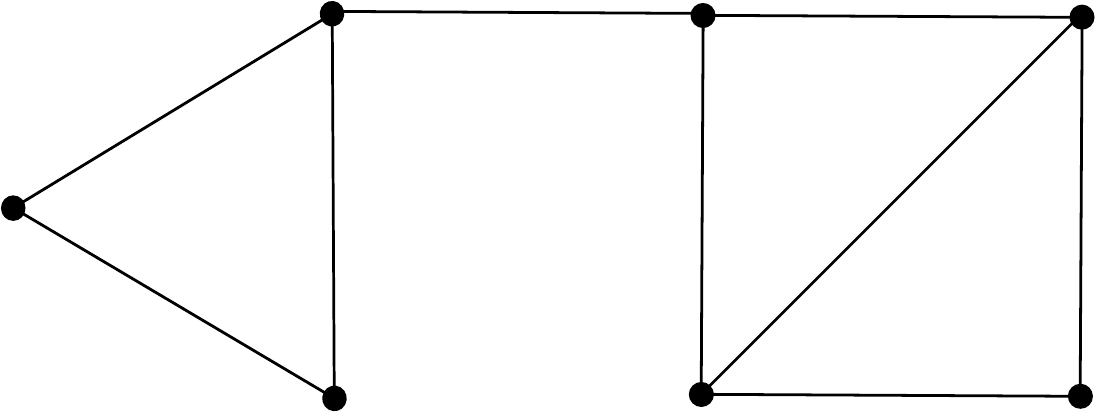}
  \caption{{\em Left:} The sublevel set $\{\varphi<\sigma_1\}$ (dashed region) associated to a potential $\varphi$ \eqref{eq:hyptopol}. The x's represent local minima, the o's, local maxima. {\em Right:} The graph associated to the
potential on the left.}
    \label{figsublevel}
\end{figure}
Among basic properties of the matrix $M_\ggg$, we recall:
\begin{itemize}
\item[-] it has a square structure $M_\ggg=\lll^*\lll$, where $\lll$ is the transpose of the incidence matrix of any oriented version of the graph $\ggg$.
In particular, $M_\ggg$ is symmetric positive.
\item[-] thanks to \eqref{eq:hyptopol} and \cite[Proposition B.1]{Mi16}, the graph $\ggg$ is connected.
\item[-] $0$  is a simple eigenvalue of $M_\ggg$
\end{itemize}

We make one more assumption which is a higher dimensional analogue of 
the hypothesis in Theorem \ref{th1}: there exist $ \mu , \nu > 0 $ such that
\begin{equation}
\label{eq:h4}
\begin{split}  
\det \varphi''(\m) &=\mu,  \ \ \ \  \forall  \m \in \mathcal U^{(0)} 
\\
\frac {\lambda_1(\s)^2}{\det\varphi''(\s)} & =-\nu, \ \ \  
\forall \s \in \mathcal U^{(1)},
 \end{split}
\end{equation}
where $\lambda_1 (\s)$ is the unique negative eigenvalue of $\varphi''(\s)$.
Assumptions \eqref{eq:h3}  and \eqref{eq:h4} can be easily removed. 
Without \eqref{eq:h4} the graph $\ggg$ is replaced by a weighted graph with a weight function depending explicitly of the values  of $\varphi''$ at critical points. 
Removing \eqref{eq:h3} leads to multigraphs in which there may 
be several edges between two
vertices. This can be also handled easily.

Assumption \eqref{eq:hyptopol} however is more fundamental and removing it 
results in major complications. 
We refer to \cite{Mi16} for results in that situation. Here we restrict ourselves to making the following

\noindent
{\bf Remark.} Under the assumption \eqref{eq:hyptopol}  the proof presented in the one dimensional case applies with relatively simple modifications. The serious difference lies in the description of $ E^{(1)} $, the eigenspace of $ \Delta_\varphi $ on one-forms, in terms of exponentially accurate quasi-modes (in one dimension it was easily done using Lemma \ref{lem:propf1}). That description is however provided by Helffer--Sj\"ostrand in the self contained Section 2.2 of \cite{HeSj85_01} -- see Theorem 2.5 there. The computation of \eqref{eq:lhat} becomes more involved and is based on the method of stationary phase -- see Helffer--Klein--Nier \cite[Proof of Proposition 6.4]{HeKlNi04_01}.

The analogue of Theorem \ref{th1} is 
\begin{theorem}\label{t:higher}
Suppose that $ \varphi $ satisfies \eqref{eq:hyp-croissancephi},\eqref{eq:hyptopol},\eqref{eq:h3} and \eqref{eq:h4}. If 
\begin{equation}
\label{eq:rho0h}
 \rho_0=\left(\frac \mu{2\pi\epsilon^2}\right)^{\frac 12} \left( \sum_{n=1}^N \beta_n 
  \indic_{E_n} + \,  r_\epsilon \right) e^{-\varphi/\epsilon^2}   , \ \ \ 
 \lim_{ \epsilon \to 0 } \Vert r_\epsilon\Vert_{L^\infty} =  0,  \ \ \beta\in 
\R^N , 
 \end{equation} 
then the solution to \eqref{eq:KS} satisfies, uniformly for $ \tau \geq 0  $,\begin{equation}\label{eq:weakconvh}
\rho( 2\epsilon^2e^{S/\epsilon^2}\tau,x)\ \rightarrow \ \sum_{n=1}^N\alpha_n(\tau)\delta_{m_n} (x) , 
\ \ \ \epsilon \to 0, 
\end{equation}
in the sense of distributions in $ x $, where  
$\alpha(t)=(\alpha_1,\ldots,\alpha_n)(\tau)$ solves 
\begin{equation}
\label{eq:dotalbis} \partial_\tau \alpha= - \kappa M_\ggg \alpha , \ \ \ \alpha ( 0 ) = \beta ,
\end{equation}
with $ M_\ggg $ is given by \eqref{eq:deflaplacegraphe} and $ \kappa =  \pi^{-1}\mu^{\frac12} \nu^{\frac12} $ with
$ \mu $ and $ \nu $ in \eqref{eq:h4}.
\end{theorem}
We also have the analogue of \eqref{eq:rhot} for any initial data.

As in the one dimensional case this theorem is a consequence of a more
precise theorem formulated using the localized states
\be\label{eq:defpsind}
\psi_n(x)=c_n(h)h^{-\frac d 4}\indic_{E_n}(x)e^{-(\varphi-\varphi_0)(x)/h}, 
\ee
where $c_n(h)$ is a normalization constant such that $\Vert \psi_n\Vert_{L^2}=1$. 
We then define a map $\Psi:\R^{N}\rightarrow L^2(\R^d)$ by 
\be\label{eq:defPsiD>1}
\Psi(\beta)=\sum_{n=1}^{N}\beta_n\psi_n, \ \ \ \forall \beta = ( \beta_1 , \ldots , \beta_{N} ) 
\in\R^{N}. 
\ee

We have the following analogue of Theorem \ref{th:DynamWitten}.

\begin{theorem}\label{th3}
Suppose  $\varphi $ satisfies \eqref{eq:hyp-croissancephi},\eqref{eq:hyptopol},\eqref{eq:h3} and \eqref{eq:h4}. There exists $C>0$ and $h_0>0$ such that for all  $\beta\in\R^{N}$ and all $0 < h < h_0 $, we have
$$
\Vert e^{-t\Delta_\varphi}\Psi(\beta)-\Psi(e^{-t\kappa\nu_h A}\beta)\Vert_{L^2}\leq Ce^{-1/Ch}, \ \ \ t\geq 0, 
$$
where $\nu_h=he^{-2S/h}$, $\kappa = \pi^{-1}\mu^{\frac12} \nu^{\frac12}$ and $A=A(h)$ is a real symmetric positive matrix having a classical expansion
$
A\sim \sum_{k=0}^\infty h^k A_k
$
and 
$A_0=M_\ggg$ with $M_\ggg$ the  Laplace matrix defined by \eqref{eq:deflaplacegraphe}.
\end{theorem}

\begin{figure}[h]
 \center
  \scalebox{0.6}{ 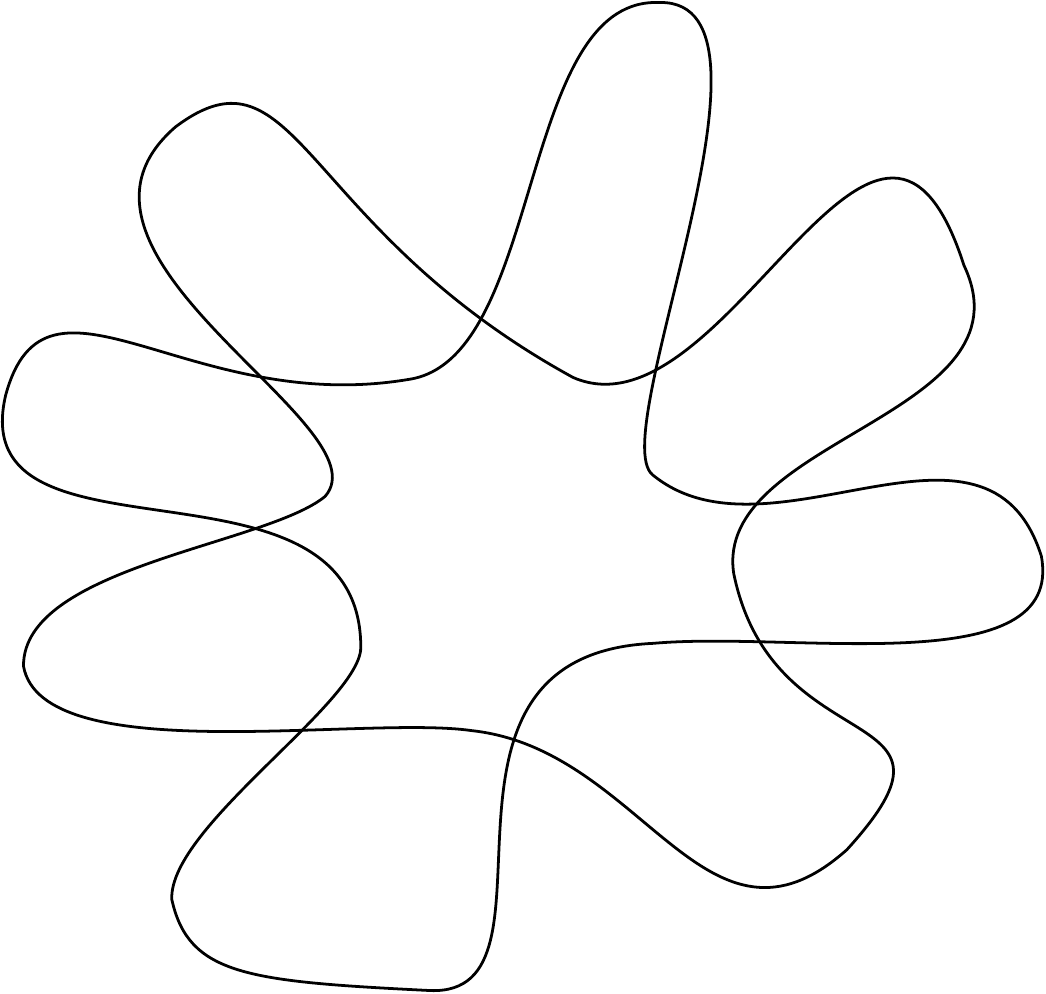} \hspace{0.1in}
  \includegraphics[scale=0.5]{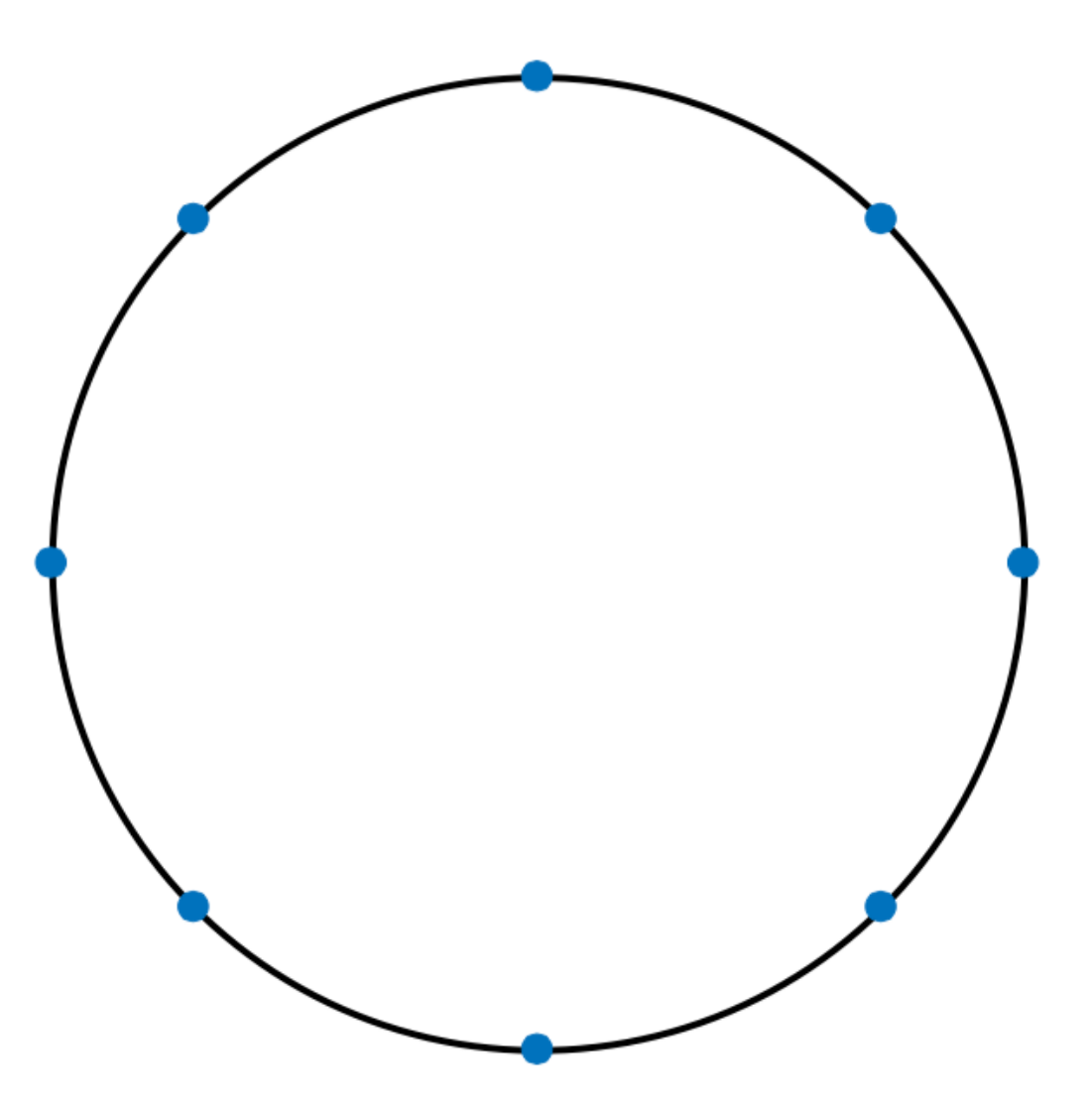}
  \caption{A two dimensional potential which is a cyclic analogue of 
the potential shown in Fig.\ref{fig1}: the corresponding 
matrix describing the Kramer--Smoluchowski evolution is given by \eqref{eq:A1}. It should compared to the matrix \eqref{eq:simpleA0} for the potential in Fig.\ref{fig1}. The corresponding cyclic graph is shown on the right.}
    \label{f2}
\end{figure}

We conclude by
one example \cite[\S 6.3]{Mi16}  for which the graph $\ggg$ is elementary. 
We assume that $ d = 2 $, $ \varphi $ has a maximum at  
$ x = 0 $, there are $ N $ minima, $ m_n $, $ N $ saddle points, $s_n $, and that \eqref{eq:hyptopol} holds -- see Fig.\ref{f2}.  We assume also that 
\[   \det \varphi'' (\m_n ) = \mu > 0 , \ \ \ \frac{\lambda_1(\s_n)}{\lambda_2(\s_n)} = - \nu < 0 , \]
where for $\s\in\uuu^{(1)}$, $\lambda_1(\s)>0>\lambda_2(\s)$ denote the two eigenvalues of $\varphi''(\s)$.

Then assumptions of Theorem \ref{th3} are satisfied. The graph $\ggg$  associated to $\varphi$ is  the cyclic graph with $N$ vertices and the corresponding  Laplacian is  given by
\begin{equation}
\label{eq:A1} \aaa_\ggg =
\begin{pmatrix} 2&\!\! \! - 1  &0&0&\ldots&\ldots&\ldots&\!\! \! - 1\\
\!\! \! - 1  &2&\!\! \! - 1  &0&\ldots&\ldots&\ldots&0\\
0&\!\! \! - 1  &2&\!\! \! - 1  &0&\ldots&\ldots&0\\
\vdots&0&\!\! \! - 1  &2&\ddots&\ddots&\ldots&0\\
\vdots&\vdots&\ddots&\ddots&\ddots&\ddots&\ddots&0\\
\vdots&\vdots&\ddots&\ddots&\ddots&\ddots&\!\! \! - 1  &0\\
0&\vdots&\ddots&\ddots&\ddots&\!\! \! - 1  &2&\!\! \! - 1  \\
\!\! \! - 1&0&\ldots&\ldots&\ldots&0&\!\! \! - 1  &2
\end{pmatrix}.
\end{equation}

\def\arXiv#1{\href{http://arxiv.org/abs/#1}{arXiv:#1}}

\bibliographystyle{amsplain}
\bibliography{ref_schmolu}

\end{document}